\pdfoutput=1

\documentclass[final]{siamltex}

\usepackage{amsfonts}
\usepackage[leqno]{amsmath}
\usepackage{ amssymb }
\usepackage{tikz}
\usepackage{algorithmic}
\usepackage{url}
\usepackage{booktabs}
\usepackage{siunitx}
\newtheorem{algorithm}{Algorithm}[section]
\newtheorem{example}{Example}[section]
\newtheorem{problem}{Problem}

\newcommand{\vect}{\operatorname{vec}}
\newcommand{\unvect}{\operatorname{unvec}}

\newcommand{\kpr}[1]{\textsuperscript{\textcircled{#1}}}    % Kronecker power
\newcommand{\ten}[1]{\mathcal{#1}}  % Tensor in Caligraphical font
\newcommand{\argmin}[1]{\underset{#1}{\operatorname{argmin}}}

\title{Symmetric Tensor Decomposition by an Iterative Eigendecomposition Algorithm}

\author{Kim Batselier \and Ngai Wong \thanks{Department of Electrical and Electronic Engineering, The University of Hong Kong, Hong Kong}}

\begin{document}

\maketitle

\begin{abstract}
We present an iterative algorithm, called the symmetric tensor eigen-rank-one iterative decomposition (STEROID), for decomposing a symmetric tensor into a real linear combination of symmetric rank-1 unit-norm outer factors using only eigendecompositions and least-squares fitting. Originally designed for a symmetric tensor with an order being a power of two, STEROID is shown to be applicable to any order through an innovative tensor embedding technique. Numerical examples demonstrate the high efficiency and accuracy of the proposed scheme even for large scale problems. Furthermore, we show how STEROID readily solves a problem in nonlinear block-structured system identification and nonlinear state-space identification.

\end{abstract}


% 15A69: Multilinear algebra, tensor products
% 15A18: Eigenvalues, singular values, and eigenvectors
% 15A23: Factorization of matrices

\pagestyle{myheadings}
\thispagestyle{plain}
\markboth{KIM BATSELIER, NGAI WONG}{STEROID}

\section{Introduction}
\label{sec:intro}
Symmetric tensors arise naturally in various engineering problems. They are especially important in the problem of blind identification of under-determined mixtures \cite{Comon2006b,Ferreol2005,Lathauwer2007}. Applications of this problem are found in areas such as speech, mobile communications, biomedical engineering and chemometrics.

The main contribution of this paper is an algorithm, called the \textbf{S}ymmetric \textbf{T}ensor \textbf{E}igen-\textbf{R}ank-\textbf{O}ne \textbf{I}terative \textbf{D}ecomposition (STEROID), that decomposes a real symmetric tensor $\ten{A}$ into a linear combination of symmetric unit-norm rank-1 tensors
\begin{align}
\nonumber \ten{A} &= l_1 \, x_1 \circ x_1 \circ \ldots \circ x_1 + \ldots + l_R \, x_R \circ x_R \circ \ldots \circ x_R,\\
&= l_1 \, x_1^d + \ldots + l_R \, x_R^d,
\label{eq:tensoreigdecomp}
\end{align}
with $l_1,\ldots ,l_R \in \mathbb{R}$ and $x_1,\ldots,x_R \in \mathbb{R}^n$. The reality of the scalar coefficients $l_1,\ldots ,l_R$ is of particular importance in the nonlinear system identification algorithm presented in Section \ref{sec:app}. The $\circ$ operation refers to the outer product, which we define in Section \ref{subsec:basics}. The notation $x_i^d \,(i=1,\ldots,R)$ denotes the $d$-times repeated outer product. In contrast to other iterative methods, STEROID does not require any initial guess and, as shown in Section \ref{sec:examples}, can handle large symmetric tensors. The minimal $R=R_{\min}$ that satisfies \eqref{eq:tensoreigdecomp} is called the symmetric rank of $\ten{A}$. More information on the rank of tensors can be found in \cite{Kolda2001,jm2012tensors} and specifically for symmetric tensors in \cite{Comon2006}. The main idea of the algorithm is to first compute a set of vectors $x_1,\ldots,x_R\, (R \geq R_{\min})$ through repeated eigendecompositions of symmetric matrices. The coefficients $l_1,\ldots,l_R$ are then found from solving a least-squares problem. STEROID was originally developed for symmetric tensors with an order that is a power of 2. It is however perfectly possible to extend the applicability of the STEROID algorithm to symmetric tensors of arbitrary order by means of an embedding procedure, which we explain in Section \ref{subsec:embed}.

In \cite{Brachat2010} an algorithm is described that decomposes a symmetric tensor over $\mathbb{C}$ using methods from algebraic geometry. This involves computing the eigenvalues of commuting matrices and as a consequence, the $l$ coefficients obtained from this method are generally complex numbers. Most attention in the literature is spent in solving the low-rank (typically rank-1) approximation problem. This problem can be formulated as follows.
\begin{problem}
Given a $d$th-order symmetric tensor $\ten{A} \in \mathbb{R}^{n \times \cdots \times n}$ and a multilinear rank $r$, find an orthogonal $n \times r$ matrix $U$ and a core tensor $\ten{S} \in \mathbb{R}^{r \times \cdots \times r}$ that minimizes the Frobenius norm
$$
||\ten{A}-\ten{S} \times_1 U \times_2 U \times_3 \cdots \times_d  U||_F,
$$
where $\times_i$ denotes the $i$th-mode product.
\label{prob:prob1}
\end{problem}

Note that the Tucker form $\ten{S} \times_1 U \times_2 U \times_3 \cdots \times_d  U$ is intrinsically different from \eqref{eq:tensoreigdecomp}, since it will also contain rank-1 terms that are not symmetric. This implies that it is not very meaningful to compare the number of rank-1 terms from the Tucker form with the number of terms computed by STEROID. Algorithms designed specifically for finding rank-1 solutions to Problem \ref{prob:prob1} are the symmetric higher-order power method (S-HOPM) \cite{regalia2002,Kofidis2000} and the shifted version of S-HOPM (SS-HOPM) \cite{Mayo2011,Mayo2014}. General low-rank algorithms are the Quasi-Newton algorithm \cite{lim2010}, the Jacobi algorithm \cite{Dooren2013} and the monotonically convergent algorithm described in \cite{Regalia2013875}.

Another common decomposition is the canonical tensor decomposition/parallel factors (CANDECOMP/PARAFAC)~\cite{candecomp, tensorreview}. This decomposition expresses a tensor as the sum of a finite number of rank-1 tensors. The tensor rank can then be defined as the minimum number of required rank-1 terms. Running a CANDECOMP algorithm such as Alternating Least Squares (ALS) on a symmetric tensor does not guarantee the symmetry of the rank-1 tensors. Other iterative methods \cite{tensorlab}, using nonlinear optimization methods, are able to guarantee the symmetry of the rank-1 terms. These methods however require the need for an initial guess and the number of computed terms also needs to be decided by the user beforehand. This is the main motivation for the development of the STEROID algorithm. STEROID is an adaptation for symmetric tensors of our earlier developed Tensor Train rank-1 SVD (TTr1SVD) algorithm \cite{ttr1svd14}, which in turn was inspired by Tensor Trains \cite{Oseledets2011}, and was an independent derivation of PARATREE \cite{Koivunen2008}. In contrast to the iterative methods mentioned above, the STEROID algorithm does not require an initial guess and the total number of terms in the decomposition follows readily from the execution of the algorithm. 

The outline of this paper is as follows. First, we define some basic notations in Section \ref{subsec:basics}. In Section \ref{sec:steroid} we fully describe our algorithm by means of a running example, together with the required embedding procedure. Two methods for the reduction of the size of the least-squares problem in the STEROID algorithm are discussed in Section \ref{sec:optimizations}. One method exploits the symmetry of the tensor, while the other method exploits the structure of the matrix in the least-squares problem. The algorithm is applied to several examples in Section \ref{sec:examples} and compared with the Jacobi algorithm \cite{Dooren2013}, Regalia's iterative method described in \cite{Regalia2013875} and the CANDECOMP-algorithm from the Tensorlab toolbox \cite{tensorlab}. In Section \ref{sec:app} we show how STEROID readily solves a problem in nonlinear block-structured system identification \cite{giri2010block} and nonlinear state-space identification \cite{paduart2010identification}. In this setting, it is often desired to recover the internal structure of an identified static nonlinear mapping \cite{schoukens2012cross,tiels2013coupled,van2013identification}. More specifically, it will be shown how STEROID can decouple a set of multivariate polynomials $f_1,\ldots,f_l$ into a collection of univariate polynomials $g_1,\ldots,g_n$, through both an affine and linear transformation.

\subsection{Tensor Notations and Basics}
\label{subsec:basics}
We will adopt the following notational conventions. A $d$th-order or $d$-way tensor, assumed real throughout this article, is a multi-dimensional array $\ten{A} \in \mathbb{R}^{n_1\times n_2 \times \ldots \times n_d}$ with elements $\ten{A}_{i_1i_2\ldots i_d}$ that can be seen as an extension of the matrix format to its general $d$th-order counterpart. Although the wordings `order' and `dimension' seem to be interchangeable in the tensor community, we prefer to call the number of indices $i_k\,(k=1,\ldots,d)$ the order of the tensor, while the maximal value $n_k\,(k=1,\ldots,d)$ associated with each index the dimension. A cubical tensor is a tensor for which $n_1=n_2=\ldots=n_d=n$.
%The $k$-mode product of a tensor $\ten{A} \in \mathbb{R}^{n_1\times n_2 \times \ldots \times n_d}$ with a matrix $U \in \mathbb{R}^{p_k \times n_k}$ is defined by
%$$
%(\ten{A} \times_k U)_{i_1\ldots i_{k-1} j_k i_{k+1} \ldots i_d} \;=\; \sum_{i_k=1}^{n_k} U_{j_k i_k} \ten{A}_{i_1 \ldots i_k\ldots i_d}.
%$$
The inner product between two tensors $\ten{A},\ten{B} \in \mathbb{R}^{n_1\times \ldots \times n_d}$ is defined as
$$
\langle \ten{A},\ten{B} \rangle \;=\; \sum_{i_1,i_2,\ldots,i_d}\,\ten{A}_{i_1i_2 \ldots i_d}\,\ten{B}_{i_1i_2 \ldots i_d}.
$$
The norm of a tensor is often taken to be the Frobenius norm $||\ten{A}||_F=\langle \ten{A},\ten{A} \rangle^{1/2}$.
A $3$rd-order rank-1 tensor $\ten{A}$ can always be written as the outer product \cite{tensorreview}
$$
\ten{A}=\lambda \, a \circ b \circ c \quad \textrm{with components } \quad \ten{A}_{i_1 i_2 i_3}\;=\;\lambda \, a_{i_1}\,b_{i_2}\,c_{i_3}
$$
with $\lambda \in \mathbb{R}$ whereas $a$, $b$ and $c$ are vectors of arbitrary lengths as demonstrated in Figure \ref{fig:outerfactor}. Similarly, any $d$-way tensor of rank 1 can be written as an outer product of $d$ vectors. %Using the $k$-mode multiplication, this outer product can also be written as $\lambda \times_1 a \times_2 b \times_3  c$ where $\lambda$ is now regarded as a $1\times 1\times 1$ tensor.

\begin{figure}[th]
\centering
\includegraphics[width=.2\textwidth]{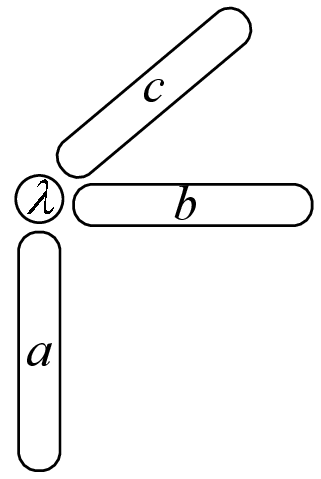}
\caption{{The outer product of 3 vectors $a,b,c$ of arbitrary lengths forming a rank-1 tensor.}}
\label{fig:outerfactor}
\end{figure}

We will only consider symmetric tensors in this article. A tensor $\ten{A}$ is symmetric if $\ten{A}_{i_1\ldots i_d} =  \ten{A}_{\pi(i_1\ldots i_d)}$, where $\pi(i_1\ldots i_d)$ is any permutation of the indices $i_1\ldots i_d$. A rank-1 symmetric $d$-way tensor $\ten{A}$ is then given by the $d$-times repeated outer product $\ten{A}=\lambda\, a \circ a \circ \ldots \circ a \triangleq \lambda a^d$. The vectorization of a tensor $\ten{A}$, denoted  
$\vect{(\ten{A})} \in \mathbb{R}^{n_1\cdots n_d}$, is the vector obtained from taking all indices together into one mode. This implies that for a symmetrical rank-1 tensor $\ten{A}$, its vectorization is
$$
\vect{(\ten{A})} \; = \; \lambda\, a \otimes a \otimes \ldots \otimes a \;=\; \lambda \,a\kpr{$d$},
$$ 
where we have introduced the shorthand notation $a\kpr{$d$}$ for the $d$-times repeated Kronecker product $\otimes$. Using the vectorization operation we can write \eqref{eq:tensoreigdecomp} as
\begin{equation}
\vect{(\ten{A})} \; =\; \lambda_1 \, x_1\kpr{$d$} + \ldots + \lambda_R \, x_R\kpr{$d$},
\label{eq:vecteneigdecomp}
\end{equation}
or equivalently as
$$
\vect{(\ten{A})} \; =\; X\;l,
$$
where $X$ is the matrix that is formed by the concatenation of all the $x_1\kpr{$d$},\ldots,x_R\kpr{$d$}$ vectors and $l \in \mathbb{R}^R$. In other words, the vectorization of a symmetric tensor $\ten{A}$ lives in the range of $X$, which is spanned by vectors $x_i\kpr{$d$}$. This requirement puts the known constraint \cite{Comon2006} on the rank of $X$
$$
\textrm{rank}(X) \leq R_{\textrm{max}} \triangleq {d+n-1 \choose n-1},
$$
for which we give a short proof in Lemma \ref{lemma:rankX}. The inverse vectorization operation $\unvect$ reshapes a vectorized tensor back into a tensor $\ten{A}=\unvect{(\vect{(\ten{A})})}$.

\section{Symmetric Tensor Eigen-Rank-One Iterative Decomposition}
\label{sec:steroid}
\subsection{Main Algorithm}
\label{subsec:mainalgo}
We now demonstrate the STEROID algorithm that decomposes a symmetric tensor into a real finite sum of symmetric rank-one outer factors by means of a 4-way tensor. Later on, we then show that STEROID is applicable to any tensor order via an innovative tensor embedding technique. The first step in the STEROID algorithm is to reshape the (4-way) symmetric $\ten{A}\in \mathbb{R}^{n\times n\times n\times n}$ into a 2-way symmetric matrix $A^{(n^2\times n^2)}$, where the bracketed superscript indicates the dimensions. The symmetry of the reshaped $A$ follows trivially from the symmetry of $\ten{A}$. Now the eigendecomposition of $A$ can be computed, which allows us to write
\begin{equation}
A\;=\; \sum_{i=1}^{n^2} \, \lambda_i \, v_i \circ v_i \;=\; \sum_{i=1}^{n^2} \, \lambda_i \, v_i \,v_i^T.
\label{eq:firsteig}
\end{equation}
The symmetry of $A$ implies that the eigenvalues $\lambda_i$ are real and the eigenvectors $v_i$ will be orthonormal. Both eigenvalues and eigenvectors can be computed by for example the symmetric QR algorithm or the divide-and-conquer method \cite{matrixcomputations}. Each of these eigenvectors $v_i$ can now be reshaped into another 2-way symmetric matrix $\bar{v}_i^{(n\times n)}$. It is readily shown that the $\bar{v}_i$ vectors are also symmetric. Specifically, the symmetry of $A$ implies that we can write
\begin{equation}
A \, P \; = \; A,
\label{eq:symmetry1}
\end{equation}
where $P$ is any permutation matrix that permutes the indices $i_ {\frac{d}{2}+1} \ldots i_{d}$. Using the eigendecomposition of $A$, we can rewrite \eqref{eq:symmetry1} as
\begin{equation}
(\lambda_1 \, v_1 v_1^T + \ldots + \lambda_{n^2}\, v_{n^2} v_{n^2}^T ) \, P \; = \; (\lambda_1 \, v_1 v_1^T + \ldots + \lambda_{n^2}\, v_{n^2} v_{n^2}^T ).
\label{eq:symmetry2}
\end{equation}
Left-multiplying \eqref{eq:symmetry2} with the eigenvector $v_i^T \, (i=1,\ldots,n^2)$ of the $i$th term, assuming $\lambda_i \neq 0$, we obtain
\begin{align*}
v_i^T\,(\lambda_1 \, v_1 v_1^T + \ldots + \lambda_{n^2}\, v_{n^2} v_{n^2}^T ) \, P &= v_i^T\,(\lambda_1 \, v_1 v_1^T + \ldots + \lambda_{n^2}\, v_{n^2} v_{n^2}^T ),\\
\Leftrightarrow \lambda_i v_i^T P &= \lambda_i v_i^T,\\
\Leftrightarrow v_i^T P &= v_i^T,
\end{align*}
which implies that any of the eigenvectors $v_i \,(i=1,\ldots,n^2)$ and consequently their reshaped $\bar{v}_i$ inhibit the same symmetry as $A$. The eigendecomposition of each of the symmetric $\bar{v}_i$'s can now also be computed. For example, $\bar{v}_1$ can then be written as
$$
\bar{v}_1 \; = \; \sum_{i=1}^n \, \lambda_{1i} \, v_{1i} \circ v_{1i},
$$ 
where the $v_{1i}$'s are again orthogonal due to the symmetry of $\bar{v}_1$. The whole procedure of repeated eigendecompositions of the reshaped eigenvectors for a $d=4, n=2$ example is depicted in Figure~\ref{fig:ex2222}.
\begin{figure}[ht]
\begin{center}
	\includegraphics[width=.9\textwidth]{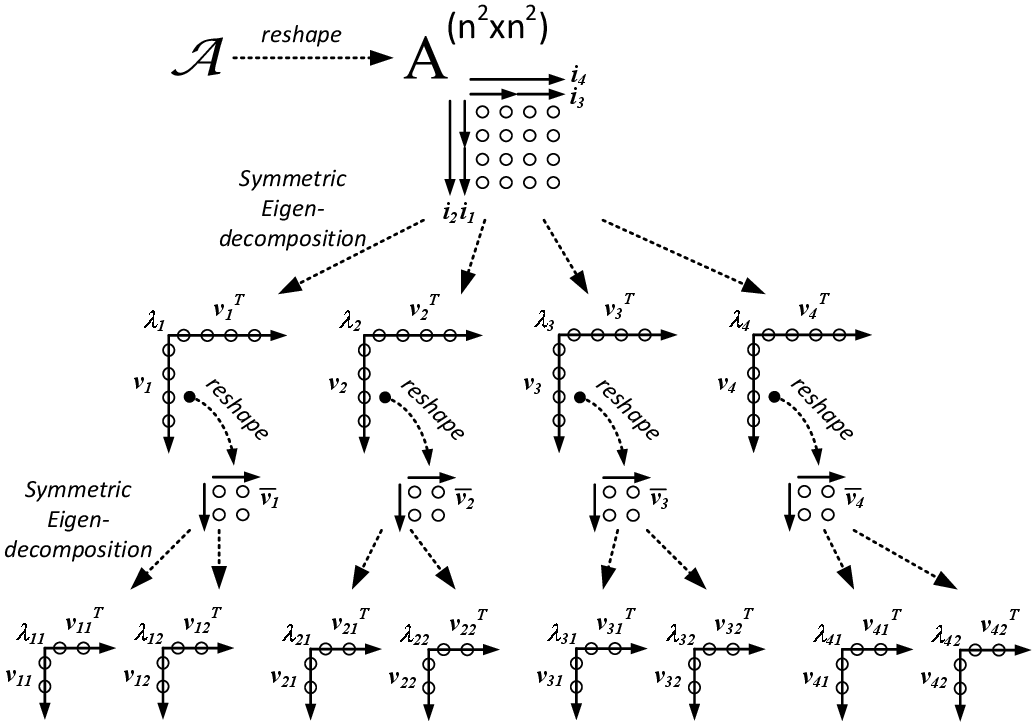}
	\end{center}
	\caption{Successive decompositions of the reshaped $A^{(n^2\times n^2)}$ for the specific case of $n=2$. Note that $\bar{v}_i$'s are always symmetric due to the $4$-way symmetry.}
	\label{fig:ex2222}
\end{figure}

Referring to Figure~\ref{fig:ex2222}, we now take the $k$th term of \eqref{eq:firsteig} and vectorize it to obtain
$$
\vect{(\lambda_k v_k v_k^T)}=\lambda_k v_k\kpr{2}.
$$
Substitution of $v_k$ by its eigendecomposition allows us to write
\begin{align}
\label{eqn:vec_vi}
\lambda_k v_k\kpr{2}=&\lambda_k \left(\lambda_{k1}v_{k1}\kpr{2}+\lambda_{k2}v_{k2}\kpr{2}\right)\otimes \left(\lambda_{k1}v_{k1}\kpr{2}+\lambda_{k2}v_{k2}\kpr{2}\right)\nonumber\\
=&\underbrace{\lambda_k \lambda_{k1}^2 v_{k1}\kpr{4}+\lambda_k \lambda_{k2}^2 v_{k2}\kpr{4}}_{h_k} + \underbrace{\lambda_k \lambda_{k1}\lambda_{k2}\left(v_{k1}\kpr{2}\otimes v_{k2}\kpr{2}+v_{k2}\kpr{2}\otimes v_{k1}\kpr{2}\right)}_{t_k}
\end{align}
where $h_k$ denotes the ``head'' part containing the \emph{pure powers} $v_{k1}\kpr{4},v_{k2}\kpr{4}$ of $v_{k1}$ and $v_{k2}$, respectively, whereas $t_k$ denotes the ``tail'' holding the sum of cross terms. Defining the head tensor $\ten H=\unvect{(\sum_k h_k)}$ and the tail tensor $\ten T=\unvect{(\sum_k t_k)}$, it can be deduced that $\ten T=\ten A-\ten H$ must also be 4-way symmetric since $\ten A$ and $\ten H$ are so. As we will explain further on, the symmetry of the tail tensor $\ten{T}$ is crucial, since it is possible that repeated eigendecompositions of the reshapings of $\ten{T}$ are also necessary in order to compute additional pure power vectors of the STEROID. The objective now is to write $\textrm{vec}(\ten{A})$ as a linear combination
$$
\vect{(\ten{A})} \;=\; l_1 x_1\kpr{4} + \ldots + l_R x_R\kpr{4}. 
$$
Good candidates for the $x_i\kpr{4}$ vectors are the pure powers that span the head part in \eqref{eqn:vec_vi}. No pure power vectors should be considered that come from an eigenvector with corresponding zero eigenvalue $\lambda_k$. Since each of the $x_i$'s is an eigenvector of a symmetric matrix, it also follows that $|| x_i\kpr{4}||_F=1$. Checking whether a decomposition as in \eqref{eq:vecteneigdecomp} exists is done by computing the residual of the following least-squares problem
\begin{equation}
\hat{l} \;=\; \argmin{l}\; ||\vect{(\ten{A})} - X \, l\,||_F,
\label{eq:ls}
\end{equation}
where $X$ is the matrix obtained from the concatenation of all the obtained pure power vectors $x_i\kpr{4}$. It is possible at this step to solve \eqref{eq:ls} with additional constraints. For example, if only positive $l_i$'s are required then one could use a reflective Newton method as described in \cite{Coleman1996}. A sparse solution $\hat{l}$ with as few nonzero $l_i$'s as possible can be computed using L1-regularization \cite{Tibshirani94}. The particular repeated Kronecker product structure for each column of $X$ results in the following upper bound on its rank.
\begin{lemma}
For the matrix $X$ in the least-squares problem \eqref{eq:ls} we have that
$$
\textrm{rank}\,(X) \leq  R_{\textrm{max}} \triangleq {d+n-1 \choose n-1}.
$$
\label{lemma:rankX}
\end{lemma}
\begin{proof}
Each column of $X$ corresponds with a vector $x_i\kpr{$d$}$, with $x_i \in \mathbb{R}^{n}$. If we label the entries of $x_i$ by $x_{i1},\ldots,x_{in}$ then $x_i\kpr{$d$}$ contains all monomials of degree $d$ in $n$ variables $x_{i1},\ldots,x_{in}$. For example, if $d=2$ and $n=2$, then
$$
x_i\kpr{2} = \begin{pmatrix}
x_{i1}\\
x_{i2}
\end{pmatrix} \otimes \begin{pmatrix}
x_{i1}\\
x_{i2}
\end{pmatrix}\; =\;  \begin{pmatrix}
x_{i1}^2\\
x_{i1}x_{i2}\\
x_{i2} x_{i1}\\
x_{i2}^2
\end{pmatrix}
$$ contains all homogeneous monomials in 2 variables of degree 2. Then for $i=1,\ldots,4$
$$
X \;=\; \begin{pmatrix}
x_{11}^2 	& x_{21}^2 		&x_{31}^2 	& x_{41}^2 \\
x_{11}x_{12}& x_{21}x_{22}	&x_{31}x_{32}& x_{41}x_{42}\\
x_{12}x_{11}& x_{22} x_{21} &x_{32}x_{31}& x_{42} x_{41}\\
x_{12}^2	& x_{22}^2		&x_{32}^2	& x_{42}^2		
\end{pmatrix}
$$
and has a rank of at most ${2+2-1 \choose 2-1}=3$, since the second and the third row are identical. For the general case there are ${d+n-1 \choose n-1}$ distinct homogeneous monomials of degree $d$ in $n$ variables and hence the rank is upper bounded by $R_{\textrm{max}}$.
\quad \end{proof}

Lemma \ref{lemma:rankX} tells us that vec($\ten{A}$) of a symmetric tensor $\ten{A}$ lives in a $R_{\textrm{max}}$-dimensional vector space. Therefore, instead of computing the STEROID, one could randomly generate $R_{\textrm{max}}$ linearly independent vectors, construct their corresponding $X$ matrix and decompose vec($\ten{A}$) along this basis. However, a random basis will most likely result in a decomposition with $R_{\textrm{max}}$ nonzero terms, while the STEROID results in a more compact decomposition, as is illustrated in the following example.

\begin{example}
We construct the following symmetric tensor $\ten{A} \;=\; l_1\,a_1^4 + l_2\,a_2^4 + l_3\,a_3^4$ with $l_1,l_2,l_3$ random real numbers and $a_1,a_2,a_3$ real random $3 \times 1$ vectors. Since $d=4,n=3$, $\textrm{vec}(\ten{A})$ lives in a ${4+3-1 \choose 3-1}= 15$-dimensional vector space $\mathcal{X}$. The STEROID of $\ten{A}$ consists of 9 nonzero terms, while the decomposition with respect to a random basis for $\mathcal{X}$ always generates 15 nonzero terms. Observe that the STEROID is not a canonical symmetric rank-1 decomposition, since the symmetric rank is by construction 3 while the STEROID consists of 9 terms.
\end{example}

From Lemma \ref{lemma:rankX} we learn two things. First, it is possible that not enough pure power vectors are computed to solve the least-squares problem \eqref{eq:ls}. In this case the residual $|| \vect{(\ten{A})} - X \, \hat{l}||_F$ will not be satisfactory and the same procedure of reshapings and eigendecompositions should be applied to the tail tensor $\ten{T}$. This will produce additional pure powers that can be used to extend $X$,  upon which one can solve the least-squares problem \eqref{eq:ls} again. Further iterations on the resulting tail tensor can be applied until a satisfactory residual is obtained. The second thing we learn is that it is also possible that $X$ becomes singular as soon as it has more than ${d+n-1 \choose n-1}$ columns. In this case it is recommended to regularize the least-squares problem such that the obtained solution is not sensitive to perturbations of the tensor $\ten{A}$. This can be done by for example computing the minimum norm solution of \eqref{eq:ls}. The whole STEROID algorithm for tensors with $d=2^k \,(k \in \mathbb{N})$ is summarized in pseudo-code in Algorithm \ref{alg:steroid}. Matlab/Octave implementations can be freely downloaded from \url{https://github.com/kbatseli/STEROID}.\\
%\\
\\
\framebox[.99\textwidth][l]{\begin{minipage}{0.99\textwidth}
\begin{algorithm}STEROID algorithm\\
\textit{\textbf{Input}}: symmetric $d$-way tensor $\ten{A}$ with $d=2^k, k\in \mathbb{N}$, tolerance $\tau$\\
\textit{\textbf{Output}}:\makebox[0pt][l]{ pure power vectors $x_i$, $\hat{l}$}
\begin{algorithmic}
\STATE $\bar{\ten{A}} \gets$ reshape $\ten{A}$ into a $n^{d/2} \times n^{d/2}$ matrix
\STATE $V_1,D_1 \gets$ eig($\bar{\ten{A}}$)
\FOR{all eigenvectors $V_1(:,i)$ with $\lambda_i \neq 0$}
\STATE recursively reshape $V_1(:,i)$ and compute its eigendecomposition
\ENDFOR
\STATE compute head tensor $\ten{H}$ from recursive eigendecompositions of $\ten{A}$
\STATE $\ten{T} \gets \ten{A}$
\STATE collect all pure powers into $X$
\STATE solve least-squares problem $\hat{l}=\argmin{l}\; ||\vect{(\ten{A})} - X \, l||_F$
\WHILE{$||\vect{(\ten{A})} - X \, \hat{l}||_2  > \tau$ AND $\textrm{rank}\,(X) < R_{\textrm{max}}$}
\STATE $\ten{T} \gets \ten{T}-\ten{H}$
\STATE \makebox[0pt][l]{add additional pure powers to $X$ by recursive eigendecompositions of $\ten{T}$}
\STATE compute new head tensor $\ten{H}$ from recursive eigendecompositions of $\ten{T}$
\STATE extend $X$ with additional pure powers
\STATE solve least-squares problem $\hat{l}=\argmin{l}\; ||\vect{(\ten{A})} - X \, l||_F$
\ENDWHILE
\end{algorithmic}
\label{alg:steroid}
\end{algorithm}
\end{minipage}}\\
\\
Every matrix from which an eigendecomposition is computed in Algorithm \ref{alg:steroid} is symmetric. This implies that the computed eigenvalues are up to a sign equal to the singular values. Hence the same kind of tolerance as for the singular values can be used to determine whether any of the $\lambda_i$'s are numerically zero \cite{matrixcomputations}. Note that a user-defined tolerance $\tau$ is required to check whether additional iterations on the tail tensor $\ten{T}$ are required.

The computational complexity of the method is dominated by the very first eigendecomposition of the $n^{d/2}\times n^{d/2}$ matrix $A$ and by solving the least-squares problem \eqref{eq:ls}. The first eigendecomposition requires a tridiagonalization of $A$, which requires $8/3\,n^{3d/2}$ flops and dominates the cost. For the actual diagonalization, QR iterations or the divide-and-conquer method can be used. The matrix $X$ in the least-squares problem also determines the computational cost. Its number of rows is $n^d$ and we can assume its number of columns to be $R_{\textrm{max}}$. Solving the least-squares problem with the SVD of $X$ then sets the maximal computational complexity to $2R_{\textrm{max}}^{2}(n^d-R_{\textrm{max}}/3)$. In Section \ref{sec:optimizations} we discuss two ways in which the size of the least-squares problem can be significantly reduced by exploiting the symmetry of $\ten{A}$ and the structure of $X$.

It is clear that the STEROID algorithm presented in Algorithm \ref{alg:steroid} only works for symmetric tensors for which the order is a power of 2. Indeed, this is a necessary requirement such that the recursive reshapings in the algorithm always lead to a square symmetric matrix. Fortunately, by employing an embedding procedure one can compute the STEROID of a symmetric tensor $\ten{A}$ of any order. We now discuss this embedding procedure in the next section.

\subsection{Tensor Embedding}
\label{subsec:embed}
The embedding procedure presented in this section allows us to compute the STEROID for a symmetric tensor of an arbitrary order. Algorithm \ref{alg:steroid} relies on the reshaping of the tensor into a square matrix and therefore the order of the tensor should be divisible by two. If the order $d$ is odd, then one can embed the tensor into a symmetric tensor of order $d+1$, reshape it again into a square matrix and continue Algorithm \ref{alg:steroid}. We now illustrate the embedding procedure with a 3-way symmetric tensor $\ten{A}$ of dimension 2 into a symmetric tensor $\ten{B}$ of order 4. Since $\ten{A}$ is symmetric, it only has 4 distinct entries, viz. $\ten{A}_{111},\ten{A}_{211}(=\ten{A}_{121}=\ten{A}_{112}),\ten{A}_{221}(=\ten{A}_{212}=\ten{A}_{122}),\ten{A}_{222}$. The idea now is to consider $\ten{A}$ as the frontal ``slice" of $\ten{B}$ in the following straightforward manner
\begin{align*}
\ten{A}_{111} &\Rightarrow \ten{B}_{1111},\\
\ten{A}_{211} &\Rightarrow \ten{B}_{2111},\\
\ten{A}_{221} &\Rightarrow \ten{B}_{2211},\\
\ten{A}_{222} &\Rightarrow \ten{B}_{2221}.
\end{align*}
In order to make sure that $\ten{B}$ is symmetric, one needs to enforce the following equalities
\begin{align*}
\ten{B}_{2111} &= \ten{B}_{1211}= \ten{B}_{1121}= \ten{B}_{1112},\\
\ten{B}_{2211} &= \ten{B}_{2121}= \ten{B}_{2112}= \ten{B}_{1221}=\ten{B}_{1212}=\ten{B}_{1122},\\
\ten{B}_{2221} &= \ten{B}_{2212}=\ten{B}_{2122}=\ten{B}_{1222}.
\end{align*}
All other entries of $\ten{B}$, in this example $\ten{B}_{2222}$, can be set to zero. We now have a symmetric $\ten{B}$ with $\ten{B}_{i_1i_2i_31}=\ten{A}_{i_1i_2i_3}$. The general embedding algorithm is presented in pseudo-code in Algorithm \ref{alg:embed}.\\
\\
\framebox[.99\textwidth][l]{\begin{minipage}{0.99\textwidth}
\begin{algorithm}symmetric tensor embedding algorithm\\
\textit{\textbf{Input}}: symmetric $d$-way cubical tensor $\ten{A}$ with $d$ an odd number\\
\textit{\textbf{Output}}:\makebox[0pt][l]{ symmetric $d+1$-way cubical tensor $\ten{B}$ with $\ten{B}_{i_1\ldots i_d1}=\ten{A}_{i_1\ldots i_d}$.}
\begin{algorithmic}
\STATE initialize $\ten{B}$ with zeros
\FOR{all nonzero $\ten{A}_{i_1\ldots i_d}$ }
\FOR{all permutations $\pi(i_1\ldots i_d 1)$}
\STATE $\ten{B}_{\pi(i_1\ldots i_d1)} \gets \ten{A}_{i_1\ldots i_d}$
\ENDFOR
\ENDFOR
\end{algorithmic}
\label{alg:embed}
\end{algorithm}
\end{minipage}}\\
\\
\\
Using Algorithm \ref{alg:embed}, it now becomes possible to adjust the STEROID algorithm such that it works for a symmetric tensor of any order. Indeed, if the order $d$ is odd, then application of Algorithm \ref{alg:embed} guarantees that the new symmetric tensor can be reshaped into a square matrix. Similarly, the obtained eigenvectors can be embedded if necessary. The following example illustrates the STEROID algorithm with embedding.
\begin{example}
Suppose we have a symmetric tensor $\ten{A} \in \mathbb{R}^{n \times \cdots \times n}$ of order $d=5$. Since $d$ is odd, it is not possible to reshape $\ten{A}$ into a square matrix. Application of Algorithm \ref{alg:embed} returns a symmetric tensor $\ten{B}$ of order $6$ such that $\ten{B}_{i_1\ldots i_d 1}=\ten{A}_{i_1\ldots i_d}$. The tensor $\ten{B}$ is then reshaped into a symmetric $n^3 \times n^3$ matrix and its eigendecomposition is computed. Each obtained eigenvector corresponding with a nonzero eigenvalue can only be reshaped into a tensor of order 3, hence the embedding has to be applied again. One can then reshape each eigenvector into a symmetric $n^2 \times n^2$ matrix and compute its eigendecomposition. Finally, the obtained eigenvectors corresponding with a nonzero eigenvalue are reshaped into a symmetric $n \times n$ matrix and the final eigendecomposition is computed, from which we obtain the pure power vectors.
\end{example}

%The total number of terms $R$ in the STEROID
%$$
%\ten{A} \; = \; \lambda_1 \, v_1 \circ v_1 \circ \ldots \circ v_1 + \ldots + \lambda_R \, v_R \circ v_R \circ \ldots \circ v_R,
%$$
%is completely determined by the total number of nonzero eigenvalues in each of the eigenvalue decompositions and %whether embeddings are required. 

%It is straightforward to derive that $N=n^{\frac{d_1}{2}}\, \ldots n^{\frac{d_k}{2}}$, with  $k=\textrm{ceil(log}_2(d))$ and where ceil($a$) rounds $a$ to the nearest integer greater than or equal to $a$. 

%The following lemma tells us how many pure power vectors are obtained from one iteration when there are no zero eigenvalues.
%\begin{lemma}
%\label{lemma:Xcol}
%Each STEROID iteration with no zero eigenvalues of a symmetric $d$-way tensor $\ten{A}^{n\times \ldots \times n}$ has $n^{d-1}$ pure power vectors.  
%\end{lemma}
%\begin{proof}
%STEROID requires that $d=2^k \, (k \in \mathbb{N})$, if none of the eigenvalues are zero then there will be $n\,n^2 \,n^{2^2} \cdots n^{2^{k-1}}=n^{1+2+2^2+\ldots + 2^{k-1}} = n^{2^k-1}=n^{d-1}$ computed eigenvectors.
%\end{proof}

\section{Reducing the size of the least-squares problem}
\label{sec:optimizations}
The scalar factors $l$ in the linear combination \eqref{eq:tensoreigdecomp} are found as the solution of the least-squares problem
$$
\hat{l} \;=\; \argmin{l}\; ||\vect{(\ten{A})} - X \, l||_F.
$$
Since each column of $X$ is a $x_i\kpr{$d$}$ vector, the total number of rows is $n^d$. The number of columns of $X$ is determined by the total number of nonzero eigenvalues in the STEROID but is in practice much less than the number of rows. The feasibility of solving the least-squares problem will therefore be largely determined by the $n^d$ number of rows of the $X$ matrix, requiring large amounts of memory. In this section we discuss two effective methods to reduce the size of $X$ and thus alleviate the memory requirement. The first method exploits the symmetry of $\ten{A}$ directly, the second method exploits the structure of the $X$ matrix to efficiently compute $X^T\,X$.

\subsection{Exploiting the symmetry of $\ten{A}$}
The symmetry of $\ten{A}$ implies that many rows of $X$ will be identical. In fact, only ${d+n-1 \choose n-1}$ rows are unique due to Lemma \ref{lemma:rankX}. If $S$ is the row selection matrix that selects the ${d+n-1 \choose n-1}$ unique rows from $X$ then \eqref{eq:ls} can be rewritten as the mathematically equivalent problem
\begin{equation}
\hat{l} \;=\; \argmin{l}\; ||S\,\vect{(\ten{A})} - S\,X \, l||_F.
\end{equation}
The number of rows of $X$ are then reduced with a factor of
$$
\gamma \triangleq \frac{n^d}{{d+n-1 \choose n-1}} \;=\; \frac{n^d \, (n-1)!}{(d+n-1) \cdots (d+1)},
$$
which grows exponential in $d$. Figure \ref{fig:gain} demonstrates the reduction factor $\gamma$ as a function of $n$ for different orders $d$. It can be seen that the reduction in number of rows first increases exponentially for small $n$ and then quickly `saturates' to an almost constant factor.

\begin{figure}[ht]
\centering
\includegraphics[width=.99\textwidth]{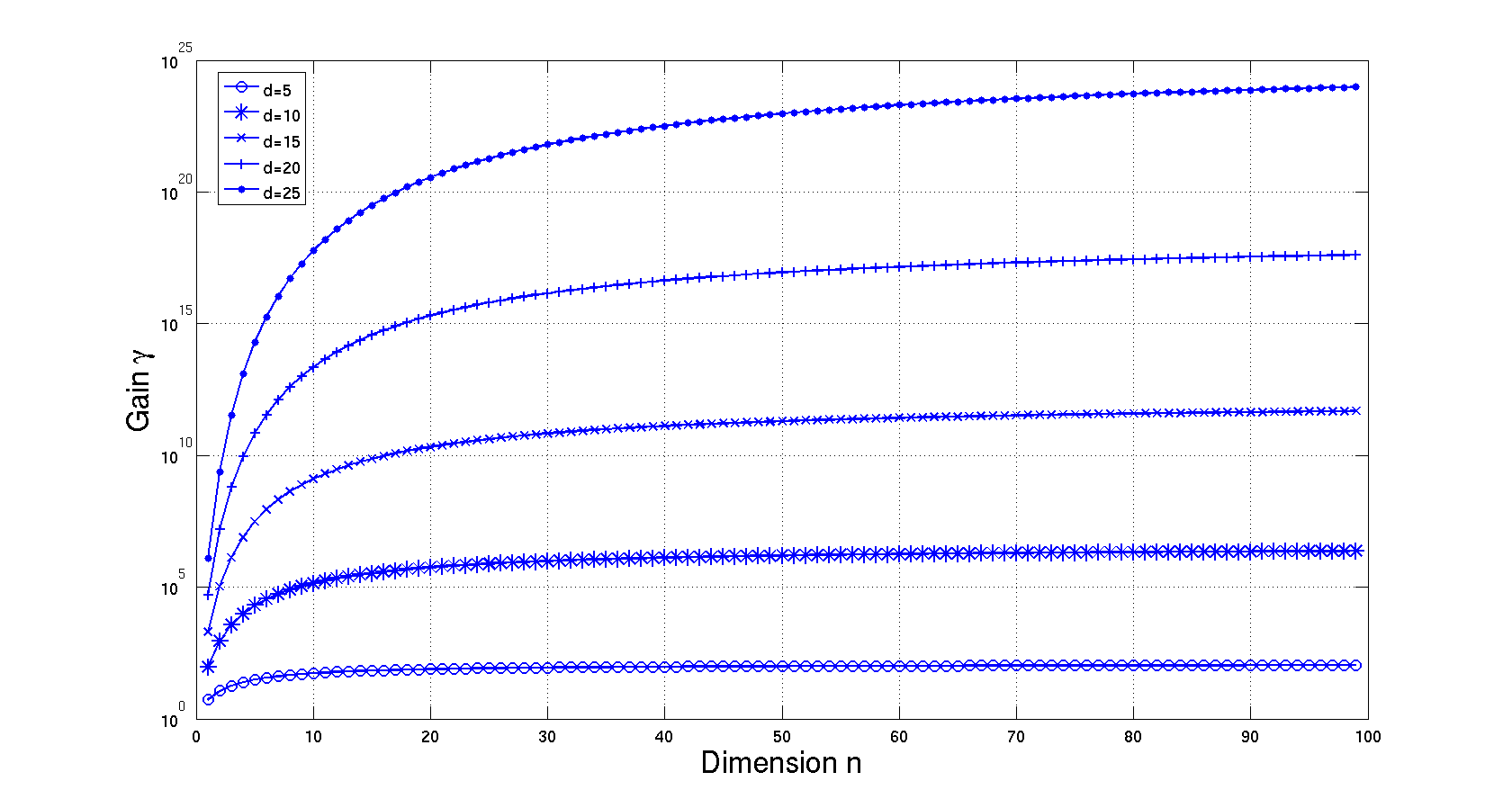}
\caption{{The reduction in number of rows of $X$ as a function of $n$ for different values of $d$.}}
\label{fig:gain}
\end{figure}

Although massive savings can be achieved by exploiting the symmetry of $\ten{A}$, $S\,X$ still has ${d+n-1 \choose n-1}= d^{n-1}/(n-1)!+O(d^{n-2})$ rows, which grows exponential in $n$. This implies that even for moderate $d$, \eqref{eq:ls} will quickly become infeasible for increasing $n$. In the next section we discuss how the size of the least-squares problem can be further reduced by exploiting the particular structure of $X$. This will come, however, at the cost of a squared condition number when solving \eqref{eq:ls}.

\subsection{Exploiting the structure of $X$}
The matrix $X$ is typically very thin, with much more rows than columns. We assume in this section that $X$ is of full column rank. One straightforward way then to reduce the size of the matrix is to left-multiply with $X^T$ to obtain
$$
X^T\, X\,l \;=\; X^T\,\textrm{vec}(\ten{A}).
$$
Since $X$ is of full column rank, $l$ will be unique and $X^T\,X$ will be symmetric and positive definite. This means that in addition to getting rid of the $n^d$ rows, only half of the entries of $X^T\,X$ need to be stored. This comes at the cost of a squared condition number $\kappa(X^T\,X)=\kappa(X)^2$, where $\kappa(X)$ denotes the condition number of the matrix $X$.

The matrix $X^T\,X$ can also be constructed without explicitly constructing $X$, what is to be avoided in the first place. Indeed, each column of $X$ is a repeated Kronecker product of a pure power vector $x_i\kpr{$d$}$. Each element of $X^T\,X$ is therefore an inner product $x_i\kpr{$d$}\,x_j\kpr{$d$}$, which can be rewritten as
$$
x_i^T x_j \otimes x_i^T x_j \otimes \ldots \otimes x_i^Tx_j \;=\; (x_i^Tx_j)^d.
$$
This allows us to construct $X^T\,X$ without the explicit construction of $X$ as described in the following lemma.
\begin{lemma}
\label{lemma:XTX}
If $V$ is the matrix that consists of the pure power vectors obtained from the STEROID, then $X^T\,X$ is constructed from $(V^T\,V).^d$, where $.^d$ denotes the entrywise operation of raising to the power $d$.
\end{lemma}

\section{Numerical examples}
\label{sec:examples}
In this section we demonstrate the STEROID algorithm on different examples. All examples were run in Matlab \cite{MATLAB:2012} on a 64-bit desktop computer with 4 cores @ 3.30 GHZ and 16 GB of memory. The first example illustrates the different steps of the STEROID algorithm on a small symmetric tensor. The second example illustrates the case where more than one STEROID iteration is required to obtain the full decomposition. In Example 3, we demonstrate the impact of the two methods to reduce the size of the least-squares problem on the residuals and run times of the STEROID algorithm. Finally, we compare the STEROID algorithm with the output of the Jacobi algorithm \cite{Dooren2013}, Regalia's iterative symmetric tensor approximation algorithm \cite{Regalia2013875} and the CANDECOMP algorithm from the Tensorlab toolbox \cite{tensorlab}. It is important to realize that in contrast to other methods, the STEROID algorithm does not require any initial guess and, as illustrated in Example 3, can handle large problems.

\subsection{Example 1: STEROID algorithm illustration on a $2\times 2\times 2$ symmetric tensor}
We first demonstrate the STEROID on a simple symmetric 3rd-order tensor
$$\ten{A}_{111}=24,\ten{A}_{211}=18,\ten{A}_{221}=12,\ten{A}_{222}=6,$$
which we first need to extend to a 4th-order symmetrical cubical tensor $\ten{B}$ using Algorithm \ref{alg:embed}. The next step of the STEROID algorithm is to reshape $\ten{B}$ into the following $4\times 4$ symmetric matrix
$$
B \;=\;
\begin{pmatrix}
24 & 18 & 18 & 12\\
18 & 12 & 12 & 6\\
18 & 12 & 12 & 6\\
12 & 6 & 6 & 0
\end{pmatrix}
$$
and compute its eigendecomposition
$$
B \;= \; V \begin{pmatrix}
-5.3939 & 0 & 0 & 0\\
0 & -6.29\times 10^{-15} & 0 & 0\\
0 & 0 & 2.64\times 10^{-15} &  0\\
0 & 0 & 0 & 53.3939\\
\end{pmatrix} \, V^T.
$$
Since $B$ has 2 eigenvalues that are numerically zero, we only need to proceed with the eigenvectors $V(:,1)$ and $V(:,4)$, where we used MATLAB notation to denote the first and fourth columns of $V$. Reshaping both $V(:,1)$ and $V(:,4)$ into a symmetric $2\times 2$ matrix and computing their eigendecomposition results in the following 4 pure power vectors
$$
x_1=\begin{pmatrix}
-0.9939\\
0.1103
\end{pmatrix},
x_2=\begin{pmatrix}
-0.1103\\
-0.9939
\end{pmatrix},
x_3=\begin{pmatrix}
-0.8396\\
-0.5431
\end{pmatrix},
x_4=\begin{pmatrix}
0.5431\\
-0.8396
\end{pmatrix}.
$$
Solving the least-squares problem
$$
\hat{l} \;=\;\argmin{l}\; ||\vect{(\ten{A})} - \begin{pmatrix}
x_1\kpr{3} & x_2\kpr{3}  & x_3\kpr{3}& x_4\kpr{3} 
\end{pmatrix}\,l||_F,
$$
results in
$$
\hat{l}\;=\;\begin{pmatrix}
3.9934\\
0.6922\\
-46.79\\
1.3916
\end{pmatrix},
$$
with a residual of $1.8546\times 10^{-14}$. Observe that the total number of terms in the computed decomposition equals the upper bound ${3+2-1 \choose 2-1}=4$. The total run time to compute the decomposition was $7.5\times 10^{-4}$ seconds. 

Another interesting example, that can be found in \cite{Comon2006}, is the symmetric tensor defined by
$$\ten{A}_{111}=-1,\ten{A}_{221}=1.$$
The STEROID algorithm computes the following decomposition
$$
\ten{A} \;=\; -2 \begin{pmatrix}
1\\
0
\end{pmatrix}^\kpr{3}
 - 1.4142 \begin{pmatrix}
-0.7071\\
0.7071
\end{pmatrix}^\kpr{3}
+1.4142 \begin{pmatrix}
0.7071\\
0.7071
\end{pmatrix}^\kpr{3}
$$
in $0.0012$ seconds. This is the same decomposition as given in \cite{Comon2006}.

%\subsection{Example 2: Comparison with other iterative methods}
%In this example we will compare the output of the STEROID algorithm with the output of the Jacobi \cite{Dooren2013} and with Regalia's iterative symmetric tensor approximation algorithm \cite{Regalia2013875} for the following symmetric tensor
%\begin{align*}
%\ten{A}_{111} &= -2.23,\\
%\ten{A}_{112} &= \ten{A}_{121}= \ten{A}_{221}= 0.38,\\
%\ten{A}_{122} &= \ten{A}_{212}=\ten{A}_{221}=-1.48,\\
%\ten{A}_{222} &= -27.5,
%\end{align*}
%for which a rank-1 approximation is desired. As described in \cite{Regalia2013875}, the Jacobi and Regalia's method start from an initial guess $u=\begin{pmatrix}0.9966 & -0.0825 \end{pmatrix}^T$ for this particular tensor and converge after 1 and 6 iterations respectively to the vector $\begin{pmatrix}-0.0522 & -0.9986\end{pmatrix}^T$, both with a residual $||\ten{A} - \bar{\ten{A}}||_F= 2.363$. The STEROID algorithm computes the following 6 vectors
%$$
%x_1=\begin{pmatrix}
%0.4725\\
%0.8813
%\end{pmatrix},
%x_2=\begin{pmatrix}
%-0.8813\\
%0.4725
%\end{pmatrix},
%x_3=\begin{pmatrix}
%0.9996\\
%-0.0267
%\end{pmatrix},
%x_4=\begin{pmatrix}
%0.0267\\
%0.9996
%\end{pmatrix},
%$$
%$$
%x_5=\begin{pmatrix}
%-0.8975\\
%-0.4410
%\end{pmatrix},
%x_6=\begin{pmatrix}
%-0.4410\\
%0.8975
%\end{pmatrix}.
%$$
%The residual $||\ten{A} - \hat{l}_i x_i^3||_F$ is minimal for $i=4$ and attains the value $2.6678$, which is relatively close to the residual of the iterative methods. The other residuals range between $18.6901$ and $27.7168$.

\subsection{Example 2: Second STEROID iteration on tail tensor $\ten{T}$}
Consider a random symmetric tensor $\ten{A} \in \mathbb{R}^{7 \times 7 \times 7 \times 7}$ with integer entries between 24 and 100. The STEROID algorithm returns 196 pure power vectors. The rank of $X$ is upper bounded by ${4+7-1 \choose 7-1}=210$ and not surprisingly we have a residual of $70.2320$, which indicates that additional pure power vectors are required. Running Algorithm \ref{alg:steroid} on the tail tensor $\ten{T}$ returns an additional 189 pure power vectors. Solving the least-squares problem \eqref{eq:ls} with all 385 pure power vectors results in a residual of $1.49\times 10^{-11}$ and 167 nonzero entries in $l$. The upper bound of 210 on the rank of $X$, together with the set of 385 pure power vectors implies that there is distinct non-uniqueness in the decomposition.

\subsection{Example 3: Comparison between original STEROID algorithm, exploitation of symmetry and using $X^TX$}
For this numerical experiment six random symmetric tensors with dimension $n=5$ and orders $d=3$ up to $d=8$ were generated. Columns two and three of Table \ref{tbl:xvectors} list the total number of computed eigendecompositions and the total number of required embeddings for each value of $d$ in the STEROID computation. Columns four and five list the total run times in seconds for computing all eigendecompositions and doing the tensor embeddings. The rightmost column lists the total time required to compute the pure power $x_i$ vectors of the STEROID and is the sum of the entries in the fourth and fifth column. From Table \ref{tbl:xvectors} it can be seen that, unless the order is a power of two, the main contribution in the total time required to compute the pure power vectors comes from the embedding procedure. The 35 embeddings for $d=6$ take about half the amount of time as the 31 embedding for $d=5$. This is explained by the fact that $d=5$ requires 1 embedding from order 5 to 6 and 30 embeddings from order 3 to 4, while for $d=6$ there are 35 embeddings from order 3 to 4. Embedding a symmetric tensor from order 5 to order 6 is a much more time-consuming process than from order 3 to 4. The largest run time is observed for $d=7$, which requires only 1 embedding from order 7 to order 8.

\begin{table}[ht]
\begin{center}
\caption{Number of eigendecompositions and embeddings and their respective total run times.}   % title of Table
\begin{tabular}{@{}rrrrrr@{}}
 & total number &  total number & total time & total time & total time\\ 
 & of eigs & of embeddings & eigs& embedding & $x_i$ vectors\\
d  &   &  & [seconds] & [seconds] & [seconds]\\ \midrule
$3$& $11$ & $1$& $0.0008$ & $0.4091$ & $0.4099$\\   [1.ex]            % inserting body of the table                            % inserts single horizontal line
 $4$& $16$ & $0$& $0.0012$ & $0$ & $0.0012$\\   [1.ex]            % inserting body of the table body of the table
                          % inserts single horizontal line
$5$& $181$ & $31$& $0.0392$ & $19.5757$ & $19.6149$\\   [1.ex]            % inserting body of the table body of the table
                            % inserts single horizontal line
 $6$& $386$ & $35$& $0.0449$ & $9.5274$ & $ 9.5714$\\   [1.ex]            % inserting body of the table of the table
                            % inserts single horizontal line
$7$& $606$ & $1$& $1.0092$ & $458.3181$ & $459.3273$\\   [1.ex]            % inserting body of the table body of the table
                            % inserts single horizontal line
 $8$& $1184$ & $0$& $0.9643$ & $0$ & $0.9643$\\   [1.ex]            % inserting body of the table body of the table
\end{tabular}
%\begin{tabular}{|c|c|c|c|c|c|}            % centered columns (4 columns)
% \hline
%\hline                        %inserts double horizontal lines
 %  & total number & total number & total time  & total time & total time \\   % inserts table
%d  & of eigs & of embeddings & eigs  & embedding & $x_i$ vectors \\   % inserts table
%  &  &  & [seconds]  &  [seconds]&  [seconds]\\ 
%heading
%\hline                              % inserts single horizontal line
% $3$& $11$ & $1$& $0.0008$ & $0.4091$ & $0.4099$\\   [1.ex]            % inserting body of the table
%\hline                              % inserts single horizontal line
% $4$& $16$ & $0$& $0.0012$ & $0$ & $0.0012$\\   [1.ex]            % inserting body of the table body of the table
% \hline                              % inserts single horizontal line
%$5$& $181$ & $31$& $0.0392$ & $19.5757$ & $19.6149$\\   [1.ex]            % inserting body of the table body of the table
%\hline                              % inserts single horizontal line
% $6$& $386$ & $35$& $0.0449$ & $9.5274$ & $ 9.5714$\\   [1.ex]            % inserting body of the table of the table
%\hline                              % inserts single horizontal line
%$7$& $606$ & $1$& $1.0092$ & $458.3181$ & $459.3273$\\   [1.ex]            % inserting body of the table body of the table
%\hline                              % inserts single horizontal line
% $8$& $1184$ & $0$& $0.9643$ & $0$ & $0.9643$\\   [1.ex]            % inserting body of the table body of the table
%\hline                           %inserts single line
%\end{tabular}
\label{tbl:xvectors}          % is used to refer this table in the text 
\end{center}
\end{table}

Once all $x_i$ vectors are computed, we solve the least-squares problem \eqref{eq:ls} in three different ways: original (no reduction of the $X$ matrix), symmetry (exploiting the symmetry of $\ten{A}$), $X^TX$ (computes $X^TX$ using Lemma \ref{lemma:XTX}). Table \ref{tbl:LS} lists the residuals $||\ten{A}-X\hat{l}||_F$ and total run times in seconds for each of the three methods. The residuals for the $X^TX$ method are only slightly worse than the other two methods due to the squared condition number. This implies that the condition numbers of $X$ were relatively small. The difference in run time is more pronounced for higher orders. Most apparent is the saving with a factor of 6942 in run time between the original and symmetry exploiting methods when $d=8$. Exploiting the symmetry reduces the total number of rows of $X$ from $5^8=390625$ down to ${8+5-1 \choose 5-1}=495$. This reduces $X$ to a $495 \times 5550$ matrix of rank 495. Using the $X^T\,X$ method when $d=8$ reduces the run time with a factor 1174 but is clearly not as good as exploiting the symmetry. The reason for this difference lies in the fact that 5550 pure power vectors are computed and therefore $X^T\,X$ is a $5550 \times 5550$ matrix, compared to the $495 \times 5550$ matrix when symmetry is exploited. 

\begin{table}[ht]
\begin{center}
\caption{Residuals and run times for solving the least-squares problem \eqref{eq:ls}.}
\begin{tabular}{@{}rrrrrrr@{}}
$d$ & \multicolumn{2}{c}{original} & \multicolumn{2}{c}{symmetry} & \multicolumn{2}{c}{$X^TX$}\\ \midrule
& $||\ten{A}-X\hat{l}||_F$ & run time & $||\ten{A}-X\hat{l}||_F$ & run time& $||\ten{A}-X\hat{l}||_F$ & run time\\
&  & [seconds] &   & [seconds] &   & [seconds]\\\midrule
$3$& \num{6.38e-15} & $0.0003$ & \num{6.45e-15}  & $0.0002$ & \num{3.71e-14}& $0.0337$\\

$4$& \num{7.25e-14} & $0.0024$ & \num{7.82e-15} & $0.0003$ & \num{5.65e-13} & $0.0256$\\

$5$& \num{5.20e-13} & $0.3866$ & \num{1.06e-14} & $0.0189$ & \num{1.91e-12}& $0.0555$\\

$6$& \num{1.86e-12} & $8.5901$ & \num{1.81e-14} & $0.0685$ & \num{2.17e-11}& $0.2721$\\

$7$& \num{2.46e-11} & $186.5544$ & \num{3.30e-14} & $0.6224$ & \num{8.36e-12}& $1.9014$\\

$8$& \num{1.42e-10} & $13312.67$ & \num{2.97e-14} & $1.9176$& \num{2.16e-10} & $11.338$\\
\end{tabular}
%\begin{tabular}{|c|c|c|c|c|c|c|}
%\hline
%$d$ & \multicolumn{2}{|c|}{original} & \multicolumn{2}{|c|}{symmetry} & \multicolumn{2}{|c|}{$X^TX$}\\
%\hline
%& $||\ten{A}-X\hat{l}||_F$ & run time & $||\ten{A}-X\hat{l}||_F$ & run time& $||\ten{A}-X\hat{l}||_F$ & run time\\
%&  & [seconds] &   & [seconds] &   & [seconds]\\
%\hline
%$3$& \num{6.38e-15} & $0.0003$ & \num{6.45e-15}  & $0.0002$ & \num{3.71e-14}& $0.0337$\\
%\hline
%$4$& \num{7.25e-14} & $0.0024$ & \num{7.82e-15} & $0.0003$ & \num{5.65e-13} & $0.0256$\\
%\hline
%$5$& \num{5.20e-13} & $0.3866$ & \num{1.06e-14} & $0.0189$ & \num{1.91e-12}& $0.0555$\\
%\hline
%$6$& \num{1.86e-12} & $8.5901$ & \num{1.81e-14} & $0.0685$ & \num{2.17e-11}& $0.2721$\\
%\hline
%$7$& \num{2.46e-11} & $186.5544$ & \num{3.30e-14} & $0.6224$ & \num{8.36e-12}& $1.9014$\\
%\hline
%$8$& \num{1.42e-10} & $13312.67$ & \num{2.97e-14} & $1.9176$& \num{2.16e-10} & $11.338$\\
%\hline
%\end{tabular}
\label{tbl:LS}
\end{center}
\end{table}

\subsection{Experiment 4: Comparison with other iterative methods}
In this numerical experiment, we apply the Jacobi \cite{Dooren2013}, Regalia's iterative symmetric tensor approximation algorithm \cite{Regalia2013875} and the CANDECOMP-algorithm of the Tensorlab toolbox \cite{tensorlab} to the symmetric tensors of Experiment 3. The Matlab implementation of the Jacobi algorithm was provided by Dr.  Mariya Ishteva. This Jacobi method implementation only works for 3rd-order tensors. We implemented Regalia's iterative algorithm and confirmed its results with the numerical experiments described in \cite{Regalia2013875}. The Tensorlab toolbox is freely available. For a given multilinear rank $r$, the Jacobi and Regalia's method return an orthogonal $n \times r$ matrix U and symmetric tensor core $\ten{S} \in \mathbb{R}^{n \times \ldots \times n}$ that minimize $ ||\ten{A}- \ten{S} \times_1 U \times_2 \cdots \times_d U||_F$. The maximal number of columns of $U$ is therefore limited to $n$. This implies that for a fully dense 3rd-order core tensor $\ten{S}$ one will have $1,4,10,20,35$ respective number of terms in the decomposition for $r=1,2,3,4,5$. Table \ref{tbl:jacobi} lists the total number of required iterations, the residual and total run time for the Jacobi method applied for all possible values of $r$. The initial orthogonal matrices to start the iterations were obtained from applying a QR orthogonalization on a random $n \times r$ matrix. The full Tucker decomposition is obtained for $r=5$ and consists of 32 nonsymmetric and 3 symmetric terms. In contrast, the STEROID consists of 50 symmetric terms and is obtained more than 3 times faster when symmetry is exploited.

\begin{table}[ht]
\begin{center}
\caption{Number of iterations, residuals and run times for the Jacobi method.}   % title of Table
\begin{tabular}{@{}rrrr@{}}
    & total number &  &  total runtime  \\   % inserts table
r  & of iterations & $ ||\ten{A}- \ten{S} \times_1 U \times_2 \cdots \times_d U||_F$ & [seconds]    \\ \midrule  % inserts table
%heading
 $1$& $38$ & $5.28$& $1.38$ \\             % inserting body of the table
 $2$& $118$ & $4.63$& $2.02$ \\             % inserting body of the table body of the table
$3$& $111$ & $3.47$& $2.46$ \\               % inserting body of the table body of the table
 $4$& $55$ & $1.83$& $1.16$ \\            % inserting body of the table of the table
$5$& $1$ & \num{7.22e-15}& $0.01$ \\   
\end{tabular}
%\begin{tabular}{|c|c|c|c|}            % centered columns (4 columns)
% \hline
%\hline                        %inserts double horizontal lines
 %  & total number &  &  total runtime  \\   % inserts table
%r  & of iterations & $ ||\ten{A}- \ten{S} \times_1 U \times_2 \cdots \times_d U||_F$ & [seconds]    \\   % inserts table
%heading
%\hline                              % inserts single horizontal line
% $1$& $38$ & $5.28$& $1.38$ \\   [1.ex]            % inserting body of the table
%\hline                              % inserts single horizontal line
% $2$& $118$ & $4.63$& $2.02$ \\   [1.ex]            % inserting body of the table body of the table
% \hline                              % inserts single horizontal line
%$3$& $111$ & $3.47$& $2.46$ \\   [1.ex]            % inserting body of the table body of the table
%\hline                              % inserts single horizontal line
% $4$& $55$ & $1.83$& $1.16$ \\   [1.ex]            % inserting body of the table of the table
%\hline                              % inserts single horizontal line
%$5$& $1$ & \num{7.22e-15}& $0.01$ \\   [1.ex]            % inserting body of the table body of the table
%\hline                              % inserts single horizontal line
%\end{tabular}
\label{tbl:jacobi}          % is used to refer this table in the text 
\end{center}
\end{table}
Regalia's iterative method is not limited to 3rd-order tensors and is therefore applied to all symmetric tensors of Experiment 3. Since we are interested in a full decomposition we set $r=5$ and therefore obtain ${3+5-1 \choose 5-1},\ldots,{8+5-1 \choose 5-1}$ terms for every respective decomposition. The initial orthogonal matrices to start the iterations were obtained from applying a QR orthogonalization on a random $n \times 5$ matrix. Table \ref{tbl:regalia} lists the total number of required iterations, the residual and total run time for each symmetric tensor from Experiment 3. Iterations were stopped when the difference in consecutive orthogonal vectors over 2 iterations was smaller than \num{1e-10}. An additional parameter $\gamma$ to ensure monotonic convergence was set to $20$. For $d=8$, the algorithm failed to finish due to a lack of sufficient memory. Compared to the symmetry exploiting STEROID algorithm, the total run time of Regalia's iterative method is up to 35 times slower.
\begin{table}[ht]
\begin{center}
\caption{Number of iterations, residuals and run times for Regalia's iterative method, $r=5$.}   % title of Table
\begin{tabular}{@{}rrrr@{}}
   & total number &  &  total runtime  \\   % inserts table
d  & of iterations & $ ||\ten{A}- \ten{S} \times_1 U \times_2 \cdots \times_d U||_F$ & [seconds]    \\   \midrule % inserts table
%heading
                                    % inserts single horizontal line
 $3$& $72$ & \num{1.78e-15}& $0.2034$ \\            % inserting body of the table
 $4$& $69$ & \num{1.24e-14}& $1.2668$\\            % inserting body of the table body of the table
$5$& $58$ & \num{4.08e-14}& $8.4662$ \\            % inserting body of the table body of the table
 $6$& $98$ & \num{3.27e-13} & $314.2224$\\              % inserting body of the table of the table
$7$& $251$ & \num{1.7373e-12} & $16301.7894$\\             % inserting body of the table body of the table
 $8$& NA & NA& NA\\   
\end{tabular}
%\begin{tabular}{|c|c|c|c|}            % centered columns (4 columns)
% \hline
%\hline                        %inserts double horizontal lines
%   & total number &  &  total runtime  \\   % inserts table
%d  & of iterations & $ ||\ten{A}- \ten{S} \times_1 U \times_2 \cdots \times_d U||_F$ & [seconds]    \\   % inserts table
%heading
%\hline                                              % inserts single horizontal line
% $3$& $72$ & \num{1.78e-15}& $0.2034$ \\   [1.ex]            % inserting body of the table
%\hline                              % inserts single horizontal line
% $4$& $69$ & \num{1.24e-14}& $1.2668$\\   [1.ex]            % inserting body of the table body of the table
% \hline                              % inserts single horizontal line
%$5$& $58$ & \num{4.08e-14}& $8.4662$ \\   [1.ex]            % inserting body of the table body of the table
%\hline                              % inserts single horizontal line
% $6$& $98$ & \num{3.27e-13} & $314.2224$\\   [1.ex]            % inserting body of the table of the table
%\hline                              % inserts single horizontal line
%$7$& $251$ & \num{1.7373e-12} & $16301.7894$\\   [1.ex]            % inserting body of the table body of the table
%\hline                              % inserts single horizontal line
 %$8$& NA & NA& NA\\   [1.ex]            % inserting body of the table body of the table
%\hline                           %inserts single line
%\end{tabular}
\label{tbl:regalia}          % is used to refer this table in the text 
\end{center}
\end{table}

Finally, the CANDECOMP-algorithm from Tensorlab is applied to all symmetric tensors of Experiment 3. This algorithm allows each rank-1 term to be symmetric as well. In Table \ref{tbl:candecomp} the total number of computed terms, the residual and total run time in seconds are listed. As with Regalia's iterative method, for $d=8$ the algorithm also fails due to the intermediate result being too large. Since we are interested in a full decomposition we set the total number of desired rank-1 terms equal to the number of terms obtained from the STEROID. This does surprisingly not result in small residuals for the CANDECOMP-algorithm. The total run time is highly variable over the different orders and for the $d=6$ case 4 times slower compared to the symmetry exploiting STEROID algorithm.

\begin{table}[ht]
\begin{center}
\caption{Number of terms, residuals and run times for Tensorlab's iterative method.}   % title of Table
\begin{tabular}{@{}rrrr@{}}
   & total number &  &  total runtime  \\   % inserts table
d  & of terms $R$ & $ ||\ten{A}-\sum_{i=1}^R a_i^d||_F $ & [seconds]    \\ \midrule  % inserts table
%heading
 $3$& $50$ & \num{ 2.7187e-09}& $0.1480$ \\             % inserting body of the table
 $4$& $75$ & \num{ 5.8317}& $9.1233$\\              % inserting body of the table body of the table
$5$& $750$ & \num{ 1.1831e-05}& $21.4996$ \\             % inserting body of the table body of the table
 $6$& $1750$ & \num{12.7366} & $1793.0454$\\          % inserting body of the table of the table
$7$& $2675$ & \num{0.0909} & $802.7768$\\            % inserting body of the table body of the table
 $8$& $5550$ & NA& NA\\   [1.ex]    
\end{tabular}
%\begin{tabular}{|c|c|c|c|}            % centered columns (4 columns)
% \hline
%\hline                        %inserts double horizontal lines
%   & total number &  &  total runtime  \\   % inserts table
%d  & of terms $R$ & $ ||\ten{A}-\sum_{i=1}^R a_i^d||_F $ & [seconds]    \\   % inserts table
%heading
%\hline                                              % inserts single horizontal line
% $3$& $50$ & \num{ 2.7187e-09}& $0.1480$ \\   [1.ex]            % inserting body of the table
%\hline                              % inserts single horizontal line
% $4$& $75$ & \num{ 5.8317}& $9.1233$\\   [1.ex]            % inserting body of the table body of the table
% \hline                              % inserts single horizontal line
%$5$& $750$ & \num{ 1.1831e-05}& $21.4996$ \\   [1.ex]            % inserting body of the table body of the table
%\hline                              % inserts single horizontal line
% $6$& $1750$ & \num{12.7366} & $1793.0454$\\   [1.ex]            % inserting body of the table of the table
%\hline                              % inserts single horizontal line
%$7$& $2675$ & \num{0.0909} & $802.7768$\\   [1.ex]            % inserting body of the table body of the table
%\hline                              % inserts single horizontal line
% $8$& $5550$ & NA& NA\\   [1.ex]            % inserting body of the table body of the table
%\hline                           %inserts single line
%\end{tabular}
\label{tbl:candecomp}          % is used to refer this table in the text 
\end{center}
\end{table}

\section{Application}
\label{sec:app}
In this section we show how STEROID readily solves a problem in nonlinear block-structured system identification \cite{giri2010block} and nonlinear state-space identification \cite{paduart2010identification}. As illustrated in Figure \ref{fig:decoupling}, the goal is to recover the internal structure of an identified static polynomial mapping
\begin{equation}
\left\{ \begin{array}{ccl}
y_1(t) &=& f_1(u_1(t),\ldots,u_p(t)),\\
\vdots \\
y_l(t) &=& f_l(u_1(t),\ldots,u_p(t)),
\end{array} \right.
\label{eq:polysys}
\end{equation}
which relates the $p$ inputs $u_1(t),\ldots,u_p(t)$ to $l$ outputs $y_1(t),\ldots,y_l(t)$. This internal structure is determined by writing each of these multivariate polynomials $f_1,\ldots,f_l$ as a linear combination of univariate polynomials $g_j(x_j)$
\begin{equation}
f_i \;=\; \sum_{j=1}^n l_{ij} \, g_j(x_j) \quad (l_{ij} \in \mathbb{R})
\label{eq:internal1}
\end{equation}
where each $x_j$ is an affine transformation of the inputs
\begin{equation}
x_j =  b_j + \sum_{k=1}^p t_{jk} u_k. \quad (b_j,t_{jk} \in \mathbb{R}).
\label{eq:internal2}
\end{equation}
The scalars $b_j$ are typically called the bias or threshold. A slightly different version of this problem has been solved in \cite{dreesen2014decoupling,usevichdecomposing}, where the conversion of the inputs $u_1,\ldots,u_p$ to the states $x_1,\ldots,x_n$ happens by means of a linear transformation.

\begin{figure}[bh]
\centering
\includegraphics[width=.99\textwidth]{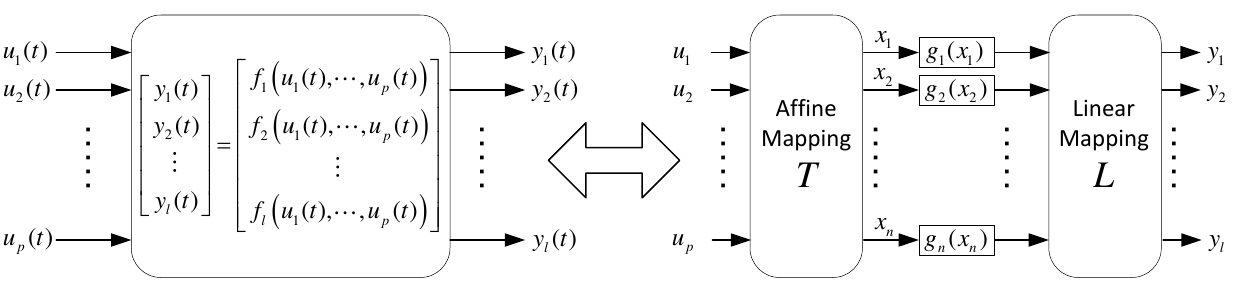}
\caption{{The polynomial mapping $f_1,\ldots,f_l$ is decoupled into a set of parallel univariate polynomials $g_1,\ldots,g_n$ by means of an affine transformation $T$ and linear transformation $L$.}}
\label{fig:decoupling}
\end{figure}

The coefficients $b_j,t_{jk}$ will turn out to be the entries of the eigenvectors $v_j$ of a STEROID and similarly the $l_{ij}$ coefficients are the real $l$ coefficients of a STEROID. We now show how this comes about. It is important to observe here that we replace our earlier notation of $x_j$ as the eigenvectors computed by the STEROID algorithm with $v_j$ in order to avoid confusion with the internal state variables $x_j$ of the nonlinear system. Let $d$ be the maximal total degree of the polynomial system \eqref{eq:polysys}. In order to estimate all coefficients $b_j,t_{jk}$, we first need to make sure that each of the polynomials \eqref{eq:polysys} is homogeneous of degree $d$. This is achieved by introducing the homogenization variable $u_0(t)$, which satisfies $u_0(t)=1\, \forall \,t\, \in \mathbb{R}$. For example, if we have $d=4$ and the polynomial $f_1=u_1^2+5u_1u_2-9$, then its homogenization is $f_1^h= u_0^2u_1^2+5u_0^2u_1u_2-9u_0^4$. The natural isomorphism between homogeneous polynomials and symmetric tensors allows us then to write the homogeneous polynomials $f_1^h,\ldots,f_l^h$ as symmetric tensors $\ten{F}_1,\ldots,\ten{F}_l$. Each of these tensors is of order $d$ and has dimension $p+1$, due to the extra homogenization variable. From the application of the STEROID Algorithm onto each symmetric tensor $\ten{F}_1,\ldots,\ten{F}_l$, a set of pure power vectors $v_j$ is obtained. From these vectors $v_j$ a basis $X$ can be constructed such that each symmetric tensor $\ten{F}_1,\ldots,\ten{F}_l$ can be decomposed in terms of this basis. The $l_{ij}$ coefficients are then found from solving the least-squares problems
$$
l_{i} \; = \;\argmin{l}\; ||\vect{(\ten{F}_i)} - X \, l||_F \quad (i=1,\ldots,l).
$$
The STEROID decomposition
$$
\ten{F}_i \; = \; \sum_{j=1}^N l_{ij} \, v_j^d,
$$
can then be written in terms of homogeneous polynomials as
\begin{equation}
f_i^h \; = \; \sum_{j=1}^N l_{ij} \, (\sum_{k=0}^p t_{jk} u_k)^d.
\label{eq:fkh}
\end{equation}
Setting the homogenization variable $u_0(t)\triangleq 1$ effectively de-homogenizes all homogeneous polynomials $f_i^h$ into
\begin{equation}
f_i \; = \; \sum_{j=1}^N l_{ij} \, (t_{i0} + \sum_{k=1}^p t_{jk} u_k)^d.
\label{eq:fi}
\end{equation}
By retaining the $n$ vectors $v_j$ corresponding with nonzero $l_{ij}$'s over all $i$'s and introducing the definitions $g_j(x_j) \triangleq x_j^d$ and $b_j \triangleq t_{i0}$ into \eqref{eq:fi}, the problem of reconstructing the internal structure of the nonlinear system as given by \eqref{eq:internal1} and \eqref{eq:internal2} is completely solved. The whole algorithm is summarized in pseudo-code in Algorithm \ref{alg:nonlinearsysid}.\\
\\
\framebox[.99\textwidth][l]{\begin{minipage}{0.99\textwidth}
\begin{algorithm}nonlinear block-structured system identification\\
\textit{\textbf{Input}}: multivariate polynomials $f_1,\ldots,f_l$\\
\textit{\textbf{Output}}:\makebox[0pt][l]{ affine transformation $T$, linear transformation $L$}
\begin{algorithmic}
\STATE $d \gets$ maximal total degree of $f_1^h,\ldots,f_l^h$
\STATE homogenize all $f_1,\ldots,f_l$ into $f_1^h,\ldots,f_l^h$ of degree $d$
\FOR{$i=1,\ldots,l$}
\STATE $\ten{F}_i \gets $ symmetric tensor corresponding with $f_i^h$
\STATE $V_i \gets $ STEROID($\ten{F}_i$)
\ENDFOR
\STATE $X \gets $ construct basis from all $V_i$ vectors
\FOR{$i=1,\ldots,l$}
\STATE $l_{i} = \argmin{l}\; ||\vect{(\ten{F}_i)} - X \, l||_F$
\ENDFOR
\STATE $L \gets $ collect all nonzero $l_{ij}$ coefficients
\STATE $T \gets $ retain only $n$ vectors from $V$ corresponding with $L$
\end{algorithmic}
\label{alg:nonlinearsysid}
\end{algorithm}
\end{minipage}}\\
\\
The following example illustrates the whole identification algorithm in detail.
\begin{example}
Consider the nonlinear 2-input-2-output system described by the polynomials $f_1,f_2$ of total degree $d=3$
$$
\begin{array}{ccl}
f_1 &=& 54u_1^3-54u_1^2u_2+8u_1^2+18u_1u_2^2+16u_1u_2-2u_2^3+8u_2^2+8u_2+1,\\
f_2 &=& -27u_1^3+27u_1^2u_2-24u_1^2-9u_1u_2^2-48u_1u_2-15u_1+u_2^3-24u_2^2-19u_2-3.
\end{array}
$$
After homogenization we obtain
$$
\begin{array}{ccl}
f_1^h &=& 54u_1^3-54u_1^2u_2+8u_0u_1^2+18u_1u_2^2+16u_0u_1u_2-2u_2^3+8u_0u_2^2\\
& &+8u_0^2u_2+u_0^3,\\
f_2^h &=& -27u_1^3+27u_1^2u_2-24u_0u_1^2-9u_1u_2^2-48u_0u_1u_2-15u_0^2u_1+u_2^3-24u_0u_2^2\\
& & -19u_0^2u_2-3u_0^3.
\end{array}
$$
The homogeneous polynomials $f_1^h$ and $f_2^h$ are converted into the symmetric third order tensors $\ten{F}_1,\ten{F}_2 \in \mathbb{R}^{3\times 3 \times 3}$. Each application of the STEROID algorithm results in 15 $v_j$ vectors, from which a basis $X$ of 30 vectors is constructed. The least-squares problem \eqref{eq:ls} is then solved for its minimum norm solution $l_{1},l_{2}$, with residuals $2.30\times 10^{-14}$ and $2.36\times 10^{-14}$ respectively. Both $l_{1}$ and $l_{2}$ contains 10 nonzero entries. Setting $u_0=1$ we obtain
$$
\begin{array}{ccl}
f_1 &=& \sum_{j=1}^{10} l_{1j} \, (b_j + \sum_{k=1}^{2} t_{jk} u_k)^3,\\ [2ex]
f_2 &=& \sum_{j=1}^{10} l_{2j} \, (b_j + \sum_{k=1}^{2} t_{jk} u_k)^3.
\end{array}
$$
\end{example}

\section{Conclusions and Remarks}
\label{sec:conclusions}
A constructive decomposition algorithm, named STEROID, has been proposed to decompose a symmetric tensor into a real linear combination of symmetric unit-norm rank-1 tensors. The method exploits symmetry and permits an efficient computation, e.g. via the symmetric QR algorithm or divide-and-conquer method, in subsequent reshapings and foldings of intermediate symmetric matrices. In contrast to other iterative methods, STEROID does not require any initial guess and and can handle large symmetric tensors. The original STEROID algorithm works with symmetric tensors whose order is a power of two, whereas an innovative tensor embedding technique is developed to remove this constraint and allows the computation of a STEROID for arbitrary orders. In addition, two methods are discussed that reduce the size of the least-squares problem, thereby increasing the feasibility to tackle large-size problems. Numerical examples have verified the high efficiency and scalability of STEROID and have demonstrated its superior performance in comparison to existing iterative methods. Finally, it was shown how STEROID can be used to decouple a set of multivariate polynomials into a collection of univariate polynomials in the setting of block-structured nonlinear system identification.

 \bibliographystyle{siam}
 \bibliography{references.bib}

\end{document}